\documentclass[10pt]{article}

\usepackage{latexsym}
\usepackage{amsmath}
\usepackage{amsthm}
\usepackage{amssymb}
\usepackage[all]{xy}
\usepackage{pgf}
\usepackage{tikz}
\usetikzlibrary{arrows,automata}

\newcommand{\II}{\mathrm{II}} \newcommand{\IM}{\mathrm{IM}} \newcommand{\IH}{\mathrm{IH}} 
\newcommand{\MI}{\mathrm{MI}} \newcommand{\MM}{\mathrm{MM}} \newcommand{\MH}{\mathrm{MH}} 
\newcommand{\HI}{\mathrm{HI}} \newcommand{\HM}{\mathrm{HM}} \newcommand{\HH}{\mathrm{HH}}
\newcommand{\IA}{\mathrm{IA}} \newcommand{\IB}{\mathrm{IB}} \newcommand{\IE}{\mathrm{IE}} 
\newcommand{\MA}{\mathrm{MA}} \newcommand{\MB}{\mathrm{MB}} \newcommand{\ME}{\mathrm{ME}} 
\newcommand{\HA}{\mathrm{HA}} \newcommand{\HB}{\mathrm{HB}} \newcommand{\HE}{\mathrm{HE}}

\newcommand{\IHb}{\mathrm{I\overline{H}}} \newcommand{\MHb}{\mathrm{M\overline{H}}} \newcommand{\HHb}{\mathrm{H\overline{H}}} \newcommand{\IEb}{\mathrm{I\overline{E}}} \newcommand{\MEb}{\mathrm{M\overline{E}}} \newcommand{\HEb}{\mathrm{H\overline{E}}}
\newcommand{\HbI}{\mathrm{\overline{H}I}} \newcommand{\HbM}{\mathrm{\overline{H}M}} \newcommand{\HbH}{\mathrm{\overline{H}H}} \newcommand{\HbHb}{\mathrm{\overline{HH}}} \newcommand{\HbA}{\mathrm{\overline{H}A}} \newcommand{\HbB}{\mathrm{\overline{H}B}} \newcommand{\HbE}{\mathrm{\overline{H}E}} \newcommand{\HbEb}{\mathrm{\overline{HE}}}

\newcommand{\N}{\mathbb{N}}
\newcommand{\Q}{\mathbb{Q}}

\newcommand{\xset}{\times \textrm{-set}}

\title{Some more notions of homomorphism-homogeneity}
\author{Deborah Lockett$^1$ and John K Truss$^2$,} 
\date{University of Leeds$^3$.}
\begin{document}
\maketitle 
\newtheorem{lemma}{Lemma}[section]
\newtheorem{theorem}[lemma]{Theorem}
\newtheorem{corollary}[lemma]{Corollary}
\newtheorem{definition}[lemma]{Definition}
\newtheorem{remark}[lemma]{Remark}
\setcounter{footnote}{1}\footnotetext{Department of Pure Mathematics, University of Leeds, Leeds LS2 9JT, UK, e-mail  d.lockett@outlook.com}
\setcounter{footnote}{2}\footnotetext{Department of Pure Mathematics, University of Leeds, Leeds LS2 9JT, UK, e-mail pmtjkt@leeds.ac.uk, corresponding author}
\setcounter{footnote}{3}\footnotetext{Supported by EPSRC grant EP/H00677X/1.}
\newcounter{number}

\begin{abstract} We extend the notion of `homomorphism-homogeneity' to a wider class of kinds of maps than previously studied, and we investigate the relations between the resulting notions of homomorphism-homogeneity, giving several examples. We also give further details on related work reported in [Deborah Lockett and John K. Truss, Generic endomorphisms of homogeneous Structures,in `Groups and model theory', Contemporary 
Mathematics 576, ed Strungmann, Droste, Fuchs, Tent, American Mathematical Society, 2012, 217-237] about the endomorphisms of chains and generic endomorphisms of trees. \end{abstract}

2010 Mathematics Subject Classification 05C99

keywords: homomorphism, endomorphism, monomorphism, embedding, epimorphism, bimorphism, homogeneous structure. 

\section{Introduction}

The notion of `homogeneity' (sometimes called `ultrahomogeneity') is an important one in model theory, having as it does a strong connection with 
quantifier elimination. In the countable case there have been a number of notable classifications of the members of certain classes which are 
homogeneous, for instance for graphs \cite{Lachlan} and partial orders \cite{Schmerl}. A generalization of this notion called 
`homomorphism-homogeneity' was introduced in \cite{Cameron2}, and developed independently by Ma\v{s}ulovi\'c \cite{Masulovic1}. This word describes a 
suite of possible definitions, stating that a structure $\cal A$ is homomorphism-homogeneous in a certain sense if any finite map from a subset of 
$\cal A$ to $\cal A$ of a particular form extends to a (totally defined) map from $\cal A$ to $\cal A$, of another specified form. The three types of 
map or partial map which were considered in \cite{Cameron2} were homomorphisms, monomorphisms, and isomorphisms. Our main goal in this paper is to 
extend the list to include three other possibilities, namely epimorphisms, bimorphisms, and embeddings, giving six possibilities in all. There are 
still only three possibilities for the partial maps however (since for instance the notion of `finite partial epimorphism' doesn't make very much 
sense for infinite structures), and so there naturally arise eighteen possible notions one might wish to consider. We shall formulate these notions, 
and show that some of them are actually the same, and others may be the same for particular classes. 

The definitions are as follows. We generally work with a countably infinite structure $\cal A$ over a finite relational language (most of our examples are graphs or partial or linear orders) and we write $A$ for the domain of $\cal A$. An {\em endomorphism} is a map $f$ from $A$ to $A$ which preserves all the relations in the signature, meaning that if $(a_1, \ldots, a_n) \in R$ then $(fa_1, \ldots, fa_n) \in R$, (a {\em homomorphism} from $A$ to $A$). Notice that this definition is very sensitive to the choice of language. For instance, strict and nonstrict linear orders, normally regarded as mere renamings of each other, work out quite differently, since if $f$ preserves $<$, then it must actually be injective, but if it is only required to preserve $\le$, it can map many points (in a convex set) to the same point. An {\em epimorphism} is a surjective endomorphism, a {\em monomorphism} is an injective endomorphism, and a {\em bimorphism} is one which is both a monomorphism and epimorphism (that is, injective and surjective). An {\em embedding} is an endomorphism which also preserves the negations of all relations (which automatically implies that it is injective, since we include $=$ as a relation), so this is the same as an {\em isomorphism} to its range, and finally an {\em automorphism} is an embedding which is also surjective.
We use symbols to denote the different kinds of maps as follows:
$\mathrm{H}$ for endomorphisms,
$\mathrm{E}$ for epimorphisms,
$\mathrm{M}$ for monomorphisms,
$\mathrm{B}$ for bimorphisms,
$\mathrm{I}$ for embeddings,
$\mathrm{A}$ for automorphisms.

For the finite approximations, there are still however only three notions, since we cannot capture surjectivity in a finite map. That is, finite partial surjective morphisms correspond to their non-surjective counterparts: finite partial endomorphisms and finite partial epimorphisms are both just homomorphisms between finite substructures; finite partial monomorphisms and bimorphisms are both finite monomorphisms; and finite partial embeddings and automorphisms are both finite isomorphisms. Thus we only consider $\mathrm{H}$, $\mathrm{M}$, $\mathrm{I}$ for the finite approximations ({\em finite partial endomorphism}, {\em finite partial monomorphism}, and {\em finite partial embedding}). There are therefore now eighteen possible notions of homomorphism-homogeneity. 
For instance, a structure $\cal A$ is {\em $\ME$} (or {\em $\ME$-homogeneous}) if every finite partial monomorphism from $\cal A$ into $\cal A$ extends to an epimorphism from $\cal A$ into $\cal A$; and $\cal A$ is $\HI$ (or {\em $\HI$-homogeneous}) if every finite partial endomorphism from $\cal A$ into $\cal A$ extends to an embedding of $\cal A$ into $\cal A$. 

These notions form a natural hierarchy inherited from that of the relation-preserving maps (see Figure~\ref{infhierpic}). 

\begin{figure}[h!tb]				
\hspace{3.9cm}
\xymatrix @!0 @dr	{
\IH \ar@{-}[rr] \ar@{-}[dd] \ar@{-}[dr] & & \MH \ar@{-}[rr] \ar@{-}[dd] \ar@{-}[dr] & & \HH \ar@{-}[dd] \ar@{-}[dr] \\
 & \IE \ar@{-}'[r][rr] \ar@{-}'[d][dd] & & \ME \ar@{-}'[r][rr] \ar@{-}'[d][dd] & & \HE \ar@{-}[dd] \\
\IM \ar@{-}[rr] \ar@{-}[dd] \ar@{-}[dr] & & \MM \ar@{-}[rr] \ar@{-}[dd] \ar@{-}[dr] & & \HM \ar@{-}[dd] \ar@{-}[dr] \\
& \IB \ar@{-}'[r][rr] \ar@{-}'[d][dd] & & \MB \ar@{-}'[r][rr] \ar@{-}'[d][dd] & &  \HB \ar@{-}[dd] \\
\II \ar@{-}[rr] \ar@{-}[dr] & & \MI \ar@{-}[rr] \ar@{-}[dr] & & \HI \ar@{-}[dr] \\
& \IA \ar@{-}[rr] & & \MA \ar@{-}[rr] & & \HA 		}
\caption{Hierarchy picture of the homomorphism-homogeneity classes for countable structures.} \label{infhierpic}
\end{figure}
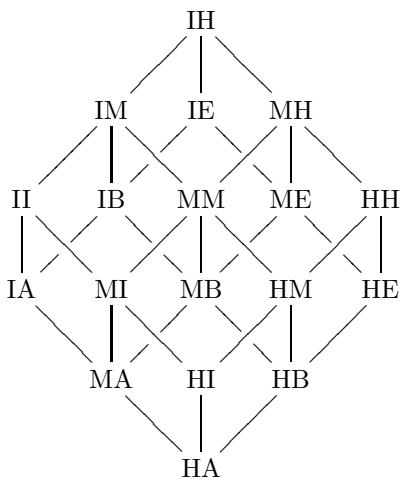

Note that previously in \cite{Cameron1, Lockett1}, $\II$ was used to denote the classical notion of homogeneity, that is that every finite partial 
automorphism (isomorphism between finite substructures) extends to an automorphism. In our new notation this actually corresponds to $\IA$, whereas 
$\II$ only says every finite partial embedding (isomorphism between finite substructures) extends to an embedding (not necessarily surjective). 
However, if we are only considering countable structures (which was indeed the case in previous papers), then we avoid any possible confusion because 
these two notions actually coincide---for countable structures, if a map extends to an embedding, then it can be extended to a surjective embedding 
(i.e.\ an automorphism).

\begin{lemma} \label{1.1}
For countable structures, $\II = \IA$, $\MI = \MA$, $\HI = \HA$.
\end{lemma}

\begin{proof}
Clearly $\IA \subseteq \II$, $\MA \subseteq \MI$, $\HA \subseteq \HI$.

We show that if countable $\cal S$ is $\II$, then in fact $\cal S$ is $\IA$. Let $p: A \to B$ be a finite partial embedding of $\cal S$ into $\cal S$ 
(that is, $p$ is an isomorphism between the finite substructures $A, B$ of $\cal S$). If $b \in \mathcal{S} \setminus B$, we show that we can extend 
$p$ to a finite partial embedding $q$ such that $b \in range(q)$. Clearly $p^{-1}$ is also a partial embedding of $\cal S$ into $\cal S$, and so since 
$\cal S$ is $\II$, $p^{-1}$ extends to an embedding $\psi$ of $\cal S$ into $\cal S$. Let $a:= \psi(b)$. Then $p^{-1} \cup \{ (b,a) \}$ is a partial 
embedding, and hence so is $p \cup \{ (a,b) \}$. Thus since $\cal S$ is countable, we can successively extend any given finite partial embedding to 
include every point in the range, and the union of these maps is a surjective embedding, that is, an automorphism. So $\cal S$ is $\IA$.

Now note that $\cal S$ is $\MI$ ($\MA$) if and only if $\cal S$ is $\II$ ($\IA$) and every finite partial monomorphism of $\cal S$ into $\cal S$ is 
an isomorphism. But since $\II$ and $\IA$ coincide by the above, $\MI$ and $\MA$ must also coincide. Similarly, $\cal S$ is $\HI$ ($\HA$) if and only 
if $\cal S$ is $\II$ ($\IA$) and every finite partial homomorphism of $\cal S$ into $\cal S$ is an isomorphism; so $\HI$ and $\HA$ also coincide.
\end{proof}

\begin{figure}[h!tb]				
\hspace{3.9cm}
\xymatrix @!0 @dr	{
\IH \ar@{-}[rr] \ar@{-}[dd] \ar@{-}[dr] & & \MH \ar@{-}[rr] \ar@{-}[dd] \ar@{-}[dr] & & \HH \ar@{-}[dd] \ar@{-}[dr] \\
 & \IE \ar@{-}'[r][rr] \ar@{-}'[d][dd] & & \ME \ar@{-}'[r][rr] \ar@{-}'[d][dd] & & \HE \ar@{-}[dd] \\
\IM \ar@{-}[rr] \ar@{-}[dr] & & \MM \ar@{-}[rr] \ar@{-}[dr] & & \HM \ar@{-}[dr] \\
& \IB \ar@{-}[rr] \ar@{-}[dd] & & \MB \ar@{-}[rr] \ar@{-}[dd] & &  \HB \ar@{-}[dd] \\
\\
& \II = \IA \ar@{-}[rr] & & \MI = \MA \ar@{-}[rr] & & \HI = \HA 	}
\caption{Modified hierarchy picture of the homomorphism-homogeneity classes for countable structures.} \label{infhierpic2}
\end{figure}
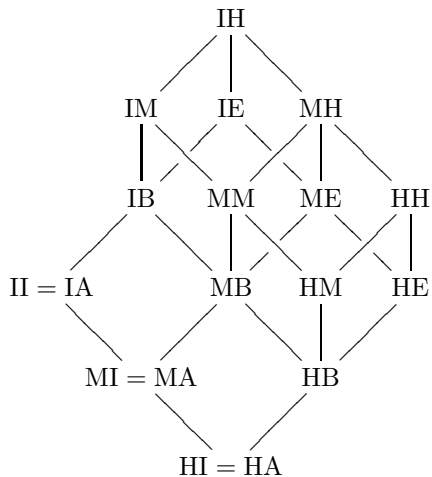

Meanwhile, none of the other new notions coincide with their non-surjective extension counterparts for all types of countable structures. We may 
demonstrate this by a single example, which is a strict partial order, taken to be a tree $T$ made up of three `branches' $A,B,C$ that are each 
isomorphic to $(\Q, <)$, such that branches $B,C$ are incomparable, and $A < B, C$. Then $T$ is $\HH$ and $\MM$ (by the classification of such countable posets, see \cite{Cameron1}), and so also $\MH$, 
$\IM$, and $\IH$. But we may easily see that $T$ is not $\IE$ (and so also not $\IB, \ME, \MB$, nor $\HE$) by considering the isomorphism $f$ that 
sends a single point $a \in A$ to a point $b \in B$. Note that every point in $T$ is comparable to $a$, and so each must be mapped to a point 
that is comparable to $b$. So no point $c \in C$ (which is incomparable to $b$) can ever be in the range of a homomorphism extension of $f$. Therefore 
there are no surjective extensions; that is, there are no epimorphisms of $T$ that extend $f$. Thus for partial orders, $\IE \subset \IH$, 
$\IB \subset \IM$, $\ME \subset \MH$, $\MB \subset \MM$, and $\HE \subset \HH$. Finally, consider a relational language with two 
binary relations, one which defines a complete graph on the points, and the other defines a partial order. Now consider the structure $T'$ in this 
relational language, which looks exactly like $T$ if we ignore the graph relation. Then $T'$ is $\MM$ since $T$ and all complete graphs are $\MM$; 
each homomorphism between substructures of $T'$ must be injective (that is, each finite homomorphism is a monomorphism), since they must preserve the 
graph relation; so $T'$ is $\HM$. (Note that $T$ itself is not $\HM$, we need the graph relation as well to ensure all homomorphisms are injective.) 
But for instance $T'$ is not $\IE$ since $T$ is not $\IE$, and so $T'$ cannot be $\HB$. Thus for relational structures in this language, 
$\HB \subset \HM$. 

If we only consider finite structures, then it is easy to see that many of the homomorphism-homogeneous classes coincide. Observe that if $f$ is an endomorphism of a finite structure $S$, which is either injective or surjective, then $f$ is actually an automorphism. So looking to extend partial maps of finite structures to epimorphisms ($\mathrm{E}$), monomorphisms ($\mathrm{M}$), bimorphisms ($\mathrm{B}$), and embedddings ($\mathrm{I}$), is the same as extending to automorphisms ($\mathrm{A}$). Thus $\mathrm{XA} = \mathrm{XI} = \mathrm{XB} = \mathrm{XM} = \mathrm{XE}$ for $\mathrm{X} \in \{ \mathrm{I, H, M} \}$, and so the hierarchy picture greatly reduces, to that shown in Figure~\ref{finhierpic}. 

\begin{figure}[h!tb]				
\hspace{4.5cm}
\xymatrix @=2.5pc @dr	{
\IH \ar@{-}[r] \ar@{-}[d] & \MH \ar@{-}[r] \ar@{-}[d] & \HH \ar@{-}[d] \\
\IA \ar@{-}[r] & \MA \ar@{-}[r] & \HA 		}
\caption{Hierarchy picture of the homomorphism-homogeneity classes, for finite structures.} \label{finhierpic}
\end{figure}
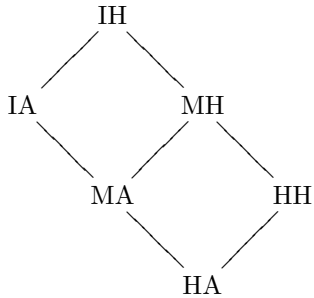

In the remainder of the paper we concentrate on three particular topics related to the general theme. In the first, in section 
2, we refine the work of \cite{Cameron1} in which the countable homomorphism homogeneous posets were described. We are able to 
say exactly which members of that classification fall into the various newly defined classes introduced here. Since Schmerl's 
list of the countable $\IA$ partial orders features prominently in this work, we here recall what his classification is. It  
comprises the following four kinds of partial order:

an antichain of size $n$, written $A_n$,

an `antichain of $n$ chains', written $B_n$, which is $n$ `copies' of the rationals 
$\mathbb Q$, which may be formally defined as $n \times {\mathbb Q}$ under the relation 
$(i, q) < (j, r) \Leftrightarrow i = j$ and $q < r$,

a chain of antichains, written $C_n$, which is $\mathbb Q$ copies of an antichain of size $n$, formally
${\mathbb Q} \times n$ under the relation $(q, i) < (r, j) \Leftrightarrow q < r$,

(where in each of these three cases, $1 \le n \le \aleph_0$), 

the generic partial order $U$, being the unique countable $\IA$ poset in which all finite partial orders embed. 

Then in section 3, we specialize to the case of linear orders, and extend work begun in \cite{Lockett2}. In that paper we were concerned with the `generic' endomorphisms, and the overall assumption was that the structure was homomorphism-homogeneous in one of the current senses, and in addition, homogeneous (meaning $\IA$-homogeneous). So for linear orders this meant that we were considering just $\mathbb Q$. We take the opportunity here of describing more fully the conjugacy classes of endomorphisms in this case and more generally (whereas in \cite{Lockett2} this was only done for generic ones). Finally in section 4 we give further details about generic endomorphisms of trees, as promised in \cite{Lockett2}.

\section{Posets}

Recall that for partial orders there are two kinds of homomorphism---strict order ($<$) preserving, and the weaker nonstrict order ($\le$) preserving. We can incorporate these into a bigger hierarchy picture of the homomorphism-homogeneous classes (see Figure~\ref{infposethierpic}), where we use $\mathrm{H, E}$ to denote strict order preserving homomorphisms and epimorphisms, and $\mathrm{\overline{H}, \overline{E}}$ to denote nonstrict order preserving homomorphisms and epimorphisms. 

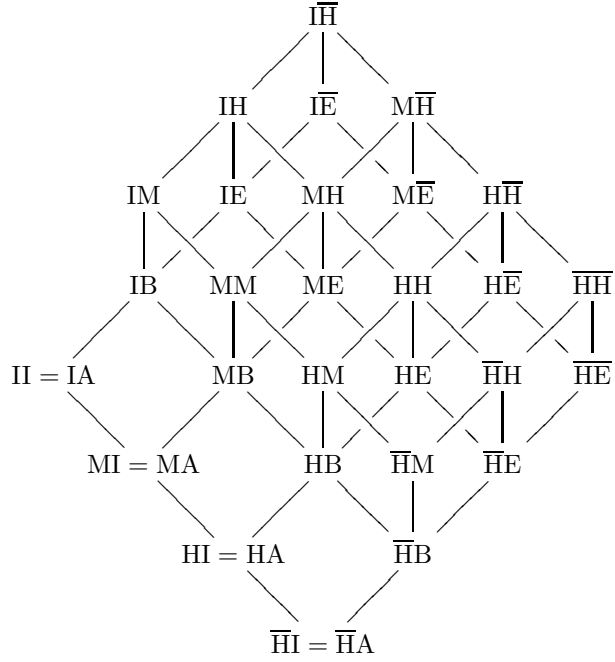
\begin{figure}[h!tb]				
\hspace{3cm}
\xymatrix @!0 @dr	{
\IHb \ar@{-}[rr] \ar@{-}[dd] \ar@{-}[dr] & & \MHb \ar@{-}[rr] \ar@{-}[dd] \ar@{-}[dr] & & \HHb \ar@{-}[rr] \ar@{-}[dd] \ar@{-}[dr] & & \HbHb \ar@{-}[dd] \ar@{-}[dr]\\
 & \IEb \ar@{-}'[r][rr] \ar@{-}'[d][dd] & & \MEb \ar@{-}'[r][rr] \ar@{-}'[d][dd] & & \HEb \ar@{-}'[r][rr] \ar@{-}'[d][dd] & & \HbEb \ar@{-}[dd] \\
\IH \ar@{-}[rr] \ar@{-}[dd] \ar@{-}[dr] & & \MH \ar@{-}[rr] \ar@{-}[dd] \ar@{-}[dr] & & \HH \ar@{-}[rr] \ar@{-}[dd] \ar@{-}[dr] & & \HbH \ar@{-}[dd] \ar@{-}[dr] \\
 & \IE \ar@{-}'[r][rr] \ar@{-}'[d][dd] & & \ME \ar@{-}'[r][rr] \ar@{-}'[d][dd] & & \HE \ar@{-}'[r][rr] \ar@{-}'[d][dd] & & \HbE \ar@{-}[dd] \\
\IM \ar@{-}[rr] \ar@{-}[dr] & & \MM \ar@{-}[rr] \ar@{-}[dr] & & \HM \ar@{-}[rr] \ar@{-}[dr] & & \HbM \ar@{-}[dr] \\
& \IB \ar@{-}[rr] \ar@{-}[dd] & & \MB \ar@{-}[rr] \ar@{-}[dd] & &  \HB \ar@{-}[rr] \ar@{-}[dd] & & \HbB \ar@{-}[dd] \\
\\
& \II = \IA \ar@{-}[rr] & & \MI = \MA \ar@{-}[rr] & & \HI = \HA \ar@{-}[rr] & & \HbI = \HbA 	}
\caption{Modified hierarchy picture of the homomorphism-homogeneity classes for countable posets.} \label{infposethierpic}
\end{figure}

The classes of countable $\IHb$ and $\IH$ posets (and other classes involving just $\mathrm{I,M,H, \overline{H}}$) were classified in \cite{Cameron1}; so now we may go through these classifications and determine the new classes. Using the notation in \cite{Cameron1}, an \emph{$\times$-set} is a poset on four elements $\{ a_1, a_2, b_1, b_2 \}$ with $a_1 \parallel a_2$, $b_1 \parallel b_2$, and $\{a_1, a_2\} < \{b_1, b_2\}$, and we say that $c$ is a \emph{midpoint} of this $\xset$ if $\{a_1, a_2\} < c < \{b_1, b_2\}$.

\begin{theorem}[Prop 25, Cor 26 from \cite{Cameron1}] \label{IHbpclass} 
A countable poset $P$ is $\IHb$ if and only if it is one of the following:
\begin{enumerate}
\item[(1)] a disjoint union of a finite or countably infinite number of incomparable countable chains (possibly of different lengths, including 
trivial chains);
\item[(2)] a tree (or inverted tree);
\item[(3)] a poset such that all finite subsets have upper and lower bounds, and every $\xset$ has a midpoint;
\item[(4)] a poset such that all finite subsets have upper and lower bounds, and no $\xset$ has a midpoint.
\end{enumerate}
Furthermore, each of these is actually $\HbHb$.
\end{theorem}

\begin{theorem}[Prop 15, Cor 16 from \cite{Cameron1}] \label{IHpclass}
A countable poset $P$ is $\IH$ if and only if it is one of the following:
\begin{enumerate}
\item[(1)] an antichain $A_n$ on $n$ points, $n$ finite or countably infinite;
\item[(2)] a disjoint union of a finite or countably infinite number of copies of $(\Q,<)$;
\item[(3)] a tree with no minimum element such that for all finite $Q \subseteq P$, $P_{(<Q)}:= \{ x \in P: x < Q \}$ has no maximal elements (or the inversion of such a tree);
\item[(4)] an extension of the generic partial order $U$ (that is, $P$ satisfies the property that for any finite $A, B \subseteq P$ with $A<B$ there is $z \in P$ with $A<z<B$);
\item[(5)] a poset such that for all finite $Q \subseteq P$, $P_{(<Q)}$ is nonempty and has no maximal elements, $P_{(>Q)} := \{ x \in P: x > Q \}$ is nonempty and has no minimal elements, and no $\xset$ has a midpoint.
\end{enumerate}
Furthermore, each of these is actually $\HH$ and $\MM$.
\end{theorem}

In part (4) here, by definition an `extension' of $U$ is a poset $P$ having the same domain, and such that $x \le y$ in $U$ implies $x \le y$ in $P$.
It can be seen (see \cite{Cameron1} Prop 13) that $P$ is an extension of $U$ if and only if for all finite $Q \subseteq P$, $P_{(<Q)}$ is nonempty and 
has no maximal elements, $P_{(>Q)}$ is nonempty and has no minimal elements, and every $\xset$ has a midpoint.

Before looking specifically at the families in these classifications, we produce some preliminary results for $\IEb$ posets. 
Recall that a poset $P$ is \emph{dense} if for all $x,y \in P$ there is $z \in P$ with $x<z<y$; and \emph{without endpoints} if there are no maximal or minimal elements (that is, for all $z \in P$ there are $x,y \in P$ with $x < z < y$). We first show that all $\IEb$ posets with some comparable points are dense and without endpoints; so in particular, apart from finite antichains, these are all infinite (note that this is not the case for $\IHb$ posets).

\begin{lemma} \label{densewoendpt}
If $P$ is a countable $\IEb$ poset which is not an antichain, then $P$ is dense and without endpoints.
\end{lemma}

\begin{proof}
Since $P$ is not an antichain, there are $a,b \in P$ with $a < b$. 
Then for each $z \in P$, the map $z \mapsto b$ must extend to a surjective map, so there is some $x \in P$ with $x < z$. Similarly for all $z \in P$ there is $y \in P$ with $z < y$. Hence $P$ is without endpoints.  
Now there are $a,b,c \in P$ with $a < b < c$. So for any $x, y \in P$ with $x < y$, the map $x \mapsto a, ~y \mapsto c$ must extend to a surjective map, so there is some $z \in P$ with $x < z < y$.
Therefore $P$ is dense.
\end{proof}

Next we see that if $P$ is a $\IEb$ poset with some incomparable points, then for each $x \in P$ there is $y \in P$ with $x \parallel y$. 

\begin{lemma} \label{xincpt}
If a countable poset $P$ which is not a chain is $\IEb$, then for each $x \in P$ there is $y \in P$ such that $x \parallel y$.
\end{lemma}

\begin{proof}
Suppose for a contradiction that there is a point $x \in P$ such that for all $z \in P$, $x,z$ are comparable (that is, $x \nparallel z$). Since $P$ is not a chain, there are points $a,b \in P$ with $a \parallel b$. Consider the isomorphism which maps $x$ to $a$. Since all points of $P$ are comparable to $x$, this map cannot be extended to a homomorphism which includes $b$ in the range; and so $P$ is not $\IEb$.
\end{proof}

\subsection{Disjoint unions of chains}

By results in \cite{Cameron1}, any disjoint union of chains is $\HbHb$, while only antichains and disjoint unions of copies of $(\Q,<)$ are $\IH$. We see that these are also the only $\IEb$ cases.

\begin{lemma}
A countable poset $P$ which is a disjoint union of chains is $\IEb$ if and only if it is either an antichain, or all chains are copies of $(\Q, <)$.
\end{lemma}

\begin{proof}
If $P$ is not an antichain, then by Lemma~\ref{densewoendpt} it is dense without endpoints. So if $P$ is a disjoint union of chains, then each chain 
is a copy of $(\Q, <)$.

Every finite or countably infinite antichain, and every disjoint union of a finite or countably infinite set of copies of $(\Q, <)$ is $\IA$, so in 
particular these are certainly $\IE$ and $\IEb$.
\end{proof}

First we determine exactly which antichains lie in which subclasses of $\IEb$. By results in \cite{Lockett1}, any countable antichain is $\MA$ and $\HbH$, but the only antichain which is $\HM$ is the trivial one, $A_1$, which is actually $\HbA$. We now show that the only nontrivial antichain which is $\HEb$ is the infinite one, which is in fact $\HbE$.

\begin{lemma}
The countably infinite antichain, $A_\omega$, is $\HbE$, but no nontrivial finite antichain is $\HEb$. 
\end{lemma}

\begin{proof}
Any finite partial endomorphism of $A_\omega$ can trivially be extended to a surjective one. 

Now suppose that $P$ is a finite nontrivial antichain, and consider a partial endomorphism $f$ such that $f(a) = f(b)$ where $a$ and $b$ are distinct members of $P$. Since $P$ is finite but $f$ is not injective, $f$ cannot be extended to a surjective map. So $P$ is not $\HEb$.
\end{proof}

So $A_1$ is in all classes; the family of nontrivial finite antichains $\{A_n: n \in \N, ~n > 1 \}$ is in $\MA$ and $\HbH$ (and all superclasses), but not in $\HM$ or $\HEb$ (or their subclasses); while $A_\omega$ is in $\MA$ and $\HbE$ (and all superclasses), but not in $\HM$ (or its subclasses).

Next we determine exactly which disjoint unions of copies of $(\Q, <)$ lie in which subclasses of $\IEb$. By results in \cite{Schmerl, Cameron1}, any disjoint union of copies of $(\Q,<)$ is $\IA, \MM, \HH, \HbHb$, but none is $\HbH$, and the only one which is $\MA$ or $\HM$ is the trival one $(\Q,<)$ which is actually $\HA$. We now show that $(\Q, <)$ is also $\HbEb$, and the only nontrivial disjoint union of copies of $(\Q,<)$ which is $\MEb$ is the infinite one, $B_\omega$, which is in fact $\MB, \HE, \HbEb$.

\begin{lemma}
The countably infinite union of copies of $(\Q,<)$ is $\MB, \HE, \HbEb$, and $(\Q, <)$ is $\HbEb$, but no nontrivial finite union of copies of $(\Q,<)$ is $\MEb$. 
\end{lemma}

\begin{proof}
Any kind of finite partial endomorphism of $B_\omega$ can be extended to a surjective one of the same kind, and similarly for any finite partial 
endomorphism of $(\Q,<)$ (in both cases, by back-and-forth, it is easy to see that we can extend to include any point in the domain, and range). 

However, now suppose that $P$ is a nontrivial finite union of copies of $(\Q,<)$. Then there are $a,b,c,d$ with $a \parallel b$ and $c<d$. Consider 
the partial monomorphism $f: a \mapsto c, ~b \mapsto d$. Then $f$ cannot be extended to a surjective map, since $P$ is a finite union of chains. So 
$P$ is not $\MEb$.
\end{proof}

So $(\Q,<)$ is in $\HA$ and $\HbEb$ (and all superclasses), but not in $\HbH$ (or subclasses); the family of nontrivial finite 
antichains of chains $\{B_n:= n \cdot (\Q, <): n \in \N, ~n > 1 \}$ is in $\IA, \MM, \HH, \HbHb$ (and superclasses), but not 
in $\MEb, \HM, \HbH$ (or their subclasses); while $B_\omega$ is in $\IA, \MB, \HE, \HbEb$ (and all superclasses), but not in 
$\MA, \HM, \HbH$ (or their subclasses).

\subsection{Trees}

A \emph{tree} (or \emph{semilinear order}) is a connected partial order $T$ such that for all $x \in T$, the substructure 
$T_{<x}:= \{ z \in T : z < x\}$ is a chain. By results in \cite{Cameron1}, any tree is $\IHb$ (in fact $\HbHb$); and a tree $T$ is $\IH$ if and only 
if it is dense, without endpoints, and for all finite $A \subseteq T$, $T_{<A} := \{ x \in T: x < A \}$ does not have a maximum point (and then $T$ is 
in fact $\HH$ and $\MM$). But by results in \cite{Schmerl, Lockett1}, no nontrivial trees (that is, ones with incomparable points) are $\IA, \HM, \HbH$.

We now show that a nontrivial tree $T$ is $\IEb$ (in fact $\HbEb$) if and only if for each $x \in T$ there is $y \in T$ with $x \parallel y$. Observe 
that for trees, this property is equivalent to saying that for all finite $A \subseteq T$, there is $y \in T$ with $A \parallel y$ (that is, for each 
$a \in A$, $a \parallel y$): since $T$ is a connected tree, for any finite $A \subseteq T$ there is a lower bound $x \in T$ with $x \le A$, so if 
$y \parallel x$, then $y \parallel A$.

\begin{lemma}
A countable nontrivial tree $T$ is $\IEb$ if and only if it is $\HbEb$, if and only if for each $x \in T$ there is $y \in T$ with $x \parallel y$.
Furthermore, $T$ is $\IE$ if and only if it is $\HE$ and $\MB$, if and only if it is $\IH$ and for each $x \in T$ there is $y \in T$ with 
$x \parallel y$.
\end{lemma}

\begin{proof}
Firstly, it is obvious that if $T$ is $\HbEb$, then $T$ is $\IEb$. Next if $T$ is $\IEb$, then by Lemma~\ref{xincpt}, the condition holds.

So for the first part, it remains to show that if a nontrivial $T$ satisfies the condition, then $T$ is $\HbEb$. We show by back-and-forth that any 
initial homomorphism between finite substructures of $T$ (that is, a finite partial endomorphism) can be extended to an epimorphism of $T$ to $T$. The 
`forth' direction is clear---since $T$ is $\HbHb$ (by Theorem~\ref{IHbpclass}), a  finite partial endomorphism can be extended to include any new 
point in the domain. It remains to show the `back' direction; that is, that any new point can be added to the range. So suppose $f: A \to B$ is a  
finite partial epimorphism, and $y \in T \setminus B$. By the condition, there is $x \in T$ with $A \parallel x$ (so note that in particular $x \notin A$), and so by defining $f(x) = y$, we can extend $f$ to a homomorphism which includes $y$ in the range. Thus since $T$ is countable, $T$ is $\HbEb$.

Similarly for the second part, it is obvious that if $T$ is $\HE$ or $\MB$, then $T$ is $\IE$. Now, if $T$ is $\IE$ then it is $\IEb$, so again by Lemma~\ref{xincpt}, the condition holds. 

Finally we show that if a nontrivial $T$ is $\IH$ and satisfies the condition, then $T$ is $\HE$ and $\MB$. As before, we show by back-and-forth that any initial homomorphism between finite substructures of $T$ can be extended to an epimorphism of $T$ to $T$; and if the initial map is injective (a monomorphism) then we can ensure that the extension map is too (a bimorphism). Again, the `forth' direction is clear---since $T$ is $\HH$ (by Theorem~\ref{IHpclass}), a  finite partial endomorphism can be extended to include any new point $x \in T$ in the domain; and similarly, since $T$ is $\MM$ (by Theorem~\ref{IHpclass}), we can ensure that in the extension $x$ is mapped to a point not already in the range. It remains to show the `back' direction. So suppose $f: A \to B$ is our partial epimorphism, and $y \in T \setminus B$. By the condition, there is $x \in T$ with $A \parallel x$ (so note that in particular $x \notin A$), and so by defining $f(x) = y$, we can extend $f$ to a homomorphism which includes $y$ in the range; and if $f$ was injective, then so is the extension. Thus since $T$ is countable, $T$ is $\HE$ and $\MB$.
\end{proof}

Note that in particular, this means that an $\IEb$ tree does not have a minimum `branching point', so it must have infinite branching (but not necessarily dense branching), and contain infinite antichains. 

The family of all trees is in $\HbHb$, but not in $\IEb$ or $\IH$.
The family of all trees such that for all finite $A \subseteq T$, there is $y \in T$ with $A \parallel y$ is in $\HbEb$, but not in $\IH$.
The family of all trees that are dense, without endpoints, such that for all finite $A \subseteq T$, $T_{<A} := \{ x \in T: x < A \}$ does not have a maximum point, is in $\MM, \HH, \HbHb$, but not in $\IEb, \HM, \HbH$.
The family of all trees that are dense, without endpoints, such that for all finite $A \subseteq T$, $T_{<A} := \{ x \in T: x < A \}$ does not have a maximum point, and there is $y \in T$ with $A \parallel y$ is in $\MB, \HE, \HbEb$, but not in $\IA, \HM, \HbH$.

Clearly, for $\mathrm{X,Y} \in \{ \mathrm{\overline{H},H,M,I,\overline{E},E,B,A} \}$ an inverted tree is $\mathrm{XY}$ if and only if it is the inversion of an $\mathrm{XY}$ tree; so the classification of such posets is exactly the same as for trees, but inverted.

\subsection{Connected posets which are not trees or inverted trees}

Let V, $\Lambda, Z$ be the 3-element posets pictured (using Hasse diagrams) in Figure~\ref{3elposets}.

\begin{figure}[h!tb]		
\hspace{5cm}
\begin{xy}
(0,15)*={\bullet}="a" ,
(10,15)*={\bullet}="b" ,
(5,5)*={\bullet}="c" ,
"a";"c" **@{-} ,
"b";"c" **@{-} ,
(5,0)*={{\rm V}},
(20,5)*={\bullet}="d" ,
(30,5)*={\bullet}="e" ,
(25,15)*={\bullet}="f" ,
"d";"f" **@{-} ,
"e";"f" **@{-} ,
(25,0)*={\Lambda},
(40,5)*={\bullet}="g" ,
(40,15)*={\bullet}="h" ,
(46,10)*={\bullet}="i" ,
"g";"h" **@{-} ,
(43,0)*={Z},
\end{xy}
\caption{The 3-element posets V, $\Lambda, Z$.} \label{3elposets}
\end{figure}
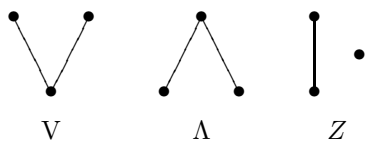

In this section we consider connected posets which are not trees or inverted trees, and these are precisely the connected posets which embed both V and 
$\Lambda$. By results in \cite{Lockett1}, there are no such posets which are $\MA, \HM, \HbH$. Meanwhile, if such a poset $P$ is $\IHb$, then it is in 
family (3) or (4) of Theorem~\ref{IHbpclass} (and so actually $\HbHb$); and if $P$ is $\IH$, then it is in family (4) or (5) of Theorem~\ref{IHpclass} 
(and so actually $\HH$ and $\MM$).

Note that unlike for trees, Lemma~\ref{xincpt} does not imply that every $\IEb$ poset $P$ satisfies the property that for all finite $A \subseteq P$ 
there is $z \in P$ with $z \parallel A$. We may observe that there are connected $\IEb$ posets with a (finite) bound on the size of antichains. We defined the family of `chains of antichains' $C_n$ for $n \in \N^* = \{1, 2, 3, \ldots\} \cup \{\aleph_0\}$ above, which are $\IA$ posets. Antichains in $C_n$ have maximum size $n$ (in fact all smaller antichains are contained in an antichain of maximum size), so for finite $n$, these are connected $\IEb$ posets with bounded antichain size. In fact, we shall later see that these are the only nontrivial connected $\IEb$ posets with bounded antichain size. 
We shall show that these are precisely the $\IEb$ posets which embed V, $\Lambda$ but not $Z$, and  afterwards consider the $\IEb$ posets which embed V, $\Lambda, Z$.

First a result about antichains in $\IEb$ posets.

\begin{lemma} \label{antichain}
If $P$ is a countable $\IEb$ poset with antichains of maximum size $n \in \N^*$, then each $x \in P$ is in an antichain of size $n$.
\end{lemma}

\begin{proof}
Otherwise suppose $A$ is an antichain of size $n$, but a maximal antichain containing $x$ has size $m<n$. Then for $a \in A$, the map $x \mapsto a$ does not extend to a surjective map.
\end{proof}

Now we consider the $\IEb$ posets which embed V, $\Lambda$ but not $Z$.

\begin{lemma}
Let $P$ be a countable $\IEb$ poset which embeds {\rm V}, $\Lambda$ but does not embed $Z$. Then $P$ is $C_n$ for some $n \in \N^*$.
\end{lemma}

\begin{proof}
If $Z$ does not embed in $P$, then $P$ is a weak order (that is, incomparability is an equivalence relation). By Lemma~\ref{antichain}, all maximal antichains of $P$ have the same size, say $n \in \N^*$ (that is, all the equivalence classes have this same size). Now by Lemma~\ref{densewoendpt}, $P$ is dense and without endpoints, and so $P$ is $C_n$.
\end{proof}

However, we may see that for finite $n$, $C_n$ is not $\MEb$; in fact we show a more general result of which this is an obvious consequence. 

\begin{lemma}
If $P$ is a countable $\IEb$ poset which is not a chain or an antichain all of whose antichains are finite, then $P$ is not $\MEb$.
\end{lemma}

\begin{proof}
Suppose $A$ is an antichain of maximum size $n \in \N, ~n>1$ containing $a_1, a_2$, and let $b<c \in P$. 
Then by Lemma~\ref{antichain}, $b$ is also in an antichain $B$ of size $n$. 
But then the monomorphism $a_1 \mapsto b, ~a_2 \mapsto c$ does not extend to a surjective map to $B$.
\end{proof}

Now we consider the $\IEb$ posets which embed {\rm V}, $\Lambda, Z$. 
In this case, we find that such a poset $P$ satisfies the property that for any finite $A \subseteq P$, there is $z \in P$ with $z \parallel A$. This property can then be used to show that all such $\IHb$ posets are $\HbEb$, and all such $\IH$ posets are $\HE$ and $\MB$.

\begin{lemma} \label{Aincpt}
Let $P$ be a countable $\IEb$ poset which embeds {\rm V}, $\Lambda, Z$. Then:
\begin{enumerate}
\item[(i)] for any $x,y \in P$ with $x<y$ there is $z \in P$ with $z \parallel \{x,y\}$;  
\item[(ii)] for any finite $A \subseteq P$, there is $z \in P$ with $z \parallel A$.
\end{enumerate}
\end{lemma}

\begin{proof}
Let $a,b,c \in P$ with $a<b$ and $c \parallel \{a,b\}$.
\begin{enumerate}
\item[(i)] Since $P$ is $\IEb$, the isomorphism $x \mapsto a, ~ y \mapsto b$ must extend to an epimorphism $f: P \to P$, so there is a point $z$ which $f$ maps onto $c$. Since $c \parallel \{a,b\}$ and $f$ is a homomorphism, $z \parallel \{x,y\}$.

\item[(ii)] By Theorem~\ref{IHbpclass}, since $P$ is $\IHb$, all finite subsets have upper and lower bounds, so there are $x,y \in P$ with 
$x \le A \le y$. By Lemma~\ref{xincpt}, we may assume that $A$ is nontrivial, so $x < y$. By (i), there is therefore some $z \in P$ with 
$z \parallel \{x,y\}$, giving $z \parallel A$ as required. 
\end{enumerate}
\end{proof}

\begin{theorem}
Let $P$ be a countable poset which embeds {\rm V}, $\Lambda, Z$.
If $P$ is $\IHb$, then $P$ is $\HbEb$. 
If $P$ is $\IH$, then $P$ is $\HE$ and $\MB$. 
\end{theorem}

\begin{proof}
First suppose that $P$ is $\IHb$. 
We show by back-and-forth that any initial homomorphism between  finite substructures of $P$ (that is, any  finite partial endomorphism) can be 
extended to an epimorphism of $P$ to $P$.
The `forth' direction is clear---if $P$ is $\IHb$ it is also $\HbHb$ (by Theorem~\ref{IHbpclass}), so a  finite partial endomorphism can be extended to include any new point in the domain. It remains to show the `back' direction; that is, that any new point can be added to the range. So suppose that $f: A \to B$ is a  finite partial epimorphism, and $y \in P \setminus B$. By Lemma~\ref{Aincpt}, there is $x \in P$ with $A \parallel x$ (so note that in particular $x \notin A$), and so by defining $f(x) = y$, we can extend $f$ to a homomorphism which includes $y$ in the range. Thus since $P$ is countable, $P$ is $\HbEb$.

Now suppose that $P$ is $\IH$. As before, we show by back-and-forth that any initial homomorphism between  finite substructures of $P$ can be extended 
to an epimorphism of $P$ to $P$; and if the initial map is injective (a monomorphism) then we can ensure that the extension map is too (a bimorphism). Again, the `forth' direction is clear---since $P$ is $\HH$ (by Theorem~\ref{IHpclass}), a  finite partial endomorphism can be extended to include any new point $x \in P$ in the domain; and similarly, since $P$ is $\MM$ (by Theorem~\ref{IHpclass}), we can ensure that in the extension $x$ is mapped to a point not already in the range. The `back' direction follows exactly as before. Suppose that $f: A \to B$ is a  finite partial epimorphism, and 
$y \in P \setminus B$. By Lemma~\ref{Aincpt}, there is $x \in P$ with $A \parallel x$ (so note that in particular $x \notin A$), and so by defining $f(x) = y$, we can extend $f$ to a homomorphism which includes $y$ in the range; and if $f$ was injective, then so is the extension. Thus since $P$ is countable, $P$ is $\HE$ and $\MB$.
\end{proof}

The family of posets $\{C_n : n \in \N, ~n > 1 \}$ is in $\IA, \MM, \HH, \HbHb$, but not in $\MEb, \HM, \HbH$.
The family of all posets such that all finite subsets have upper and lower bounds, and either every $\xset$ has a midpoint or none does, is in $\HbHb$, but not in $\IEb$ or $\IH$.
The family of all posets such that for all finite $A \subseteq P$, $A$ has upper and lower bounds and there is $z \in P$ with $A \parallel z$, and either every $\xset$ has a midpoint or none does, is in $\HbEb$, but not in $\IH$.
The family of all posets such that for all finite $A \subseteq P$, $P_{(<A)}$ is nonempty and has no maximal elements, $P_{(>A)}$ is nonempty and has no minimal elements, and either every $\xset$ has a midpoint or none does, is in $\MM, \HH, \HbHb$, but not in $\IEb, \HM, \HbH$.
The family of all posets such that for all finite $A \subseteq P$, $P_{(<A)}$ is nonempty and has no maximal elements, $P_{(>A)}$ is nonempty and has no minimal elements, and there is $z \in P$ with $A \parallel z$, and either every $\xset$ has a midpoint or none does, is in $\MB, \HE, \HbEb$, but not in $\IA, \HM, \HbH$. Finally, note that $C_\omega$ is in the final family so it is $\MB, \HE, \HbEb$, but it is also $\IA$, but not $\MA, \HM, \HbH$.

\subsection{Classifications and hierarchy picture}

Since we have now determined exactly which classes all of the homomorphism-homogeneous posets lie in, we have full classifications and can determine 
the final hierarchy picture of the classes. This is shown in Figure~\ref{finalposethierpic}. Extra labels have been added to show which posets lie in 
each of the classes. Individual posets and families of posets are shown below the classes of which they are a member (in boxes joined to the classes 
by a dotted line). The individual posets $A_1, ~A_\omega, ~B_1 = (\Q, <), ~B_\omega, ~C_\omega$ are clear. We use $A_n, ~B_n, ~C_n$ to represent the 
families of such posets for $n \in \N, ~n > 1$. New notation is introduced for the remaining families of homomorphism-homogeneous posets. 
A poset $P$ satisfies the \emph{bounding property} (BP) if either every $\xset$ has a midpoint or none does (or vacuously there are no $\xset$s); and if V embeds in $P$, then all finite subsets of $P$ have lower bounds; and similarly if $\Lambda$ embeds in $P$, then all finite subsets of $P$ have upper bounds.
A poset $P$ satisfies the \emph{density property} (DP) if it does not have maximal or minimal elements; and if V embeds in $P$, then for all finite $Q \subseteq P$, $P_{(<Q)}$ is nonempty and has no maximal elements; and similarly if $\Lambda$ embeds in $P$, then for all finite $Q \subseteq P$, $P_{(>Q)}$ is nonempty and has no minimal elements. 
A poset $P$ satisfies the \emph{incomparable point property} (IPP) if for any finite $A \subseteq P$, there is $z \in P$ with $z \parallel A$. 
Let $\mathcal{P}_{b}$ be the family of all posets which satisfy (BP) (note that this family contains all trees and all inverted trees, and the families (3), (4) of Theorem~\ref{IHbpclass}). 
Let $\mathcal{P}_{bd}$ be the family of all posets which satisfy (BP) and (DP) (so this family contains the families (3), (4), (5) of Theorem~\ref{IHpclass}). 
Let $\mathcal{P}_{bi}$ be the family of all posets which satisfy (BP) and (IPP); and let $\mathcal{P}_{bdi}$ be the family of all posets which satisfy (BP), (DP) and (IPP).

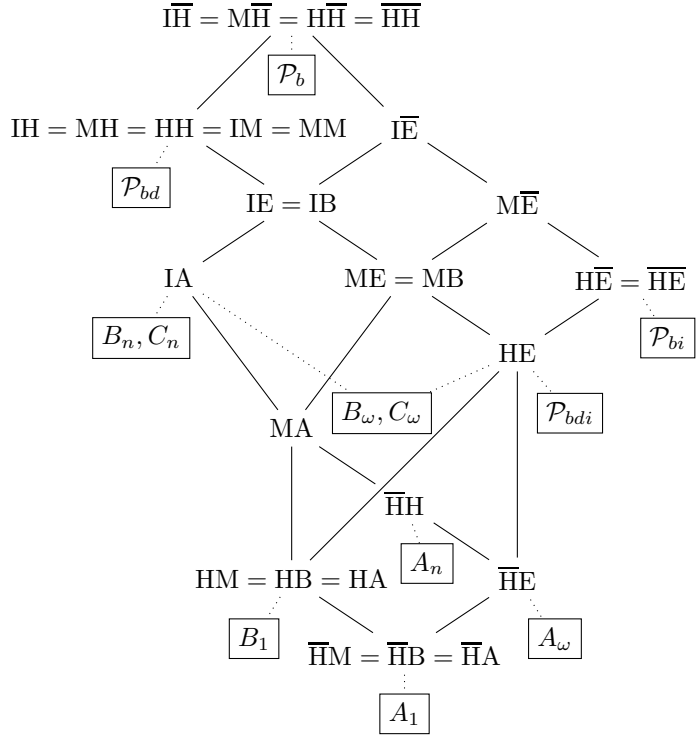
\begin{figure}[h!tb]				
\hspace{2cm}
\begin{tikzpicture}
\draw 	(3.5,10.5) node (IHb) {$\IHb = \MHb = \HHb = \HbHb$}
	(3.5,9.7) node[draw] (Pb) {$\mathcal{P}_{b}$}
	(2,9) node (IH) {$\IH = \MH = \HH = \IM = \MM$}
	(1.5,8.2) node[draw] (Pbd) {$\mathcal{P}_{bd}$}
	(5,9) node (IEb) {$\IEb$}
	(3.5,8) node (IE) {$\IE = \IB$}
	(6.5,8) node (MEb) {$\MEb$}
	(2,7) node (IA) {$\IA$}
	(1.5,6.2) node[draw] (Bn) {$B_n, C_n$}
	(5,7) node (ME) {$\ME = \MB$}
	(8,7) node (HEb) {$\HEb = \HbEb$}
	(8.5,6.2) node[draw] (Pbi) {$\mathcal{P}_{bi}$}
	(6.5,6) node (HE) {$\HE$}
	(7.2,5.2) node[draw] (Pbdi) {$\mathcal{P}_{bdi}$}
	(4.7,5.2) node[draw] (Bw) {$B_\omega, C_\omega$}
	(3.5,5) node (MA) {$\MA$}
	(5,4) node (HbH) {$\HbH$}
	(5.3,3.2) node[draw] (An) {$A_n$}
	(3.5,3) node (HM) {$\HM = \HB = \HA$}
	(3,2.2) node[draw] (Q) {$B_1$}
	(6.5,3) node (HbE) {$\HbE$}
	(7,2.2) node[draw] (Aw) {$A_\omega$}
	(5,2) node (HbM) {$\HbM = \HbB = \HbA$}
	(5,1.2) node[draw] (A1) {$A_1$};
\draw 	(IHb) -- (IH)
	(IHb) -- (IEb)
	(IH) -- (IE)
	(IEb) -- (IE)
	(IEb) -- (MEb)
	(IE) -- (IA)
	(IE) -- (ME)
	(MEb) -- (ME)
	(MEb) -- (HEb)
	(IA) -- (MA)
	(ME) -- (MA)
	(ME) -- (HE)
	(HEb) -- (HE)
	(MA) -- (HbH)
	(MA) -- (HM)
	(HE) -- (HbE)
	(HE) -- (HM)
	(HbH) -- (HbE)
	(HM) -- (HbM)
	(HbE) -- (HbM);
\draw[dotted] 	(IHb) -- (Pb)
	(IH) -- (Pbd)
	(HEb) -- (Pbi)
	(IA) -- (Bn)
	(IA) -- (Bw)
	(HE) -- (Bw)
	(HE) -- (Pbdi)
	(HbH) -- (An)
	(HM) -- (Q)
	(HbE) -- (Aw)
	(HbM) -- (A1);
\end{tikzpicture}
\caption{Final hierarchy picture of the homomorphism-homogeneity classes for countable posets.} \label{finalposethierpic}
\end{figure}

\section{Linear orders}

In this section we give some more information about endomorphisms of linear orders. We already described generic endomorphisms of $\mathbb Q$ up to 
conjugacy in \cite{Lockett2}, but these ideas can be used to characterize conjugacy classes in arbitrary linear orders, generalizing 
Holland's method for automorphisms (see \cite{Glass} Theorem 2.2.5). Briefly, Holland's method is as follows. If $g$ is an automorphism of a linear 
order, then the convex closures of orbits are called `orbitals' (also sometimes called `intervals of support'), and these come in three possible 
kinds, depending whether $g$ is increasing, decreasing, or constant on the orbital (in the last case, the orbital must be a singleton), and we assign 
parities $+1$, $-1$, and 0 is these cases respectively. The family of orbitals then becomes a $\{\pm 1, 0\}$-coloured chain, called the 
{\em orbital pattern} of $g$, and the main result about these is that two automorphisms of a doubly homogeneous chain are conjugate if and only if 
they have isomorphic orbital patterns.

If we consider endomorphisms instead, then it still makes sense to talk about their being conjugate, but since inversion is involved, the conjugacy 
must be taken as an automorphism, even though the endomorphisms considered need not be. This for instance is the way to describe the iterative 
behaviour of such maps. In \cite{Lockett2}, the generic endomorphisms of the ordered rationals were described in terms of their conjugacy class, and 
we wish here to expand on those remarks, and apply them to other linear orders. As usual we should distinguish between maps which preserve $<$, or 
just $\le$---it was the latter which was treated in \cite{Lockett2}. These are however closely related, since preserving $<$ rather than $\le$ amounts 
to restricting to just two out of the four possible monoids, so it actually suffices to consider just $\le$, and to read off the corresponding results 
for $<$ by looking at those for $\le$ in just Emb and Aut, and not Hom and Epi.

Let us therefore consider a linearly ordered set $(X, \le)$ under the reflexive relation, and let $f$ be an endomorphism. The analogue of an 
{\em orbital} of $x \in X$ as described above should be taken to be the convex closure of $\bigcup_{n \in {\mathbb Z}}f^nx$ (where as usual, by $f^nx$ 
where $n = -m$ is negative we understand $\{y: f^my = x\}$), and such a set may also be characterized as an equivalence class under the relation 
$\sim$ given by $x \sim y$ if for some $k, l, m \in {\mathbb N}$, $f^k(x) \le f^l(y) \le f^m(x)$. This relation is clearly reflexive. To see that it 
is symmetric, note that if $f^k(x) \le f^l(y) \le f^m(x)$ then $f^{k+l}(y) \le f^{k+m}(x) \le f^{l+m}(y)$, and it is transitive, since if 
$f^{k_1}(x) \le f^{l_1}(y) \le f^{m_1}(x)$ and $f^{k_2}(y) \le f^{l_2}(z) \le f^{m_2}(y)$ then 
$f^{k_1+k_2}(x) \le f^{l_1+k_2}(y) \le f^{l_1+l_2}(z) \le f^{l_1+m_2}(y) \le f^{m_1+m_2}(x)$. It is immediate that its equivalence classes are convex. 
To see that they are the convex closures of sets of the form $\bigcup_{n \in {\mathbb Z}}f^nx$ note that the convex closure of 
$\bigcup_{n \in {\mathbb Z}}f^nx$ consists of all elements $y$ such that $x \sim y$, in other words this is the $\sim$-class of $x$.

The main difference here from the situation for automorphisms is that there are infinitely many different behaviours that can arise on an orbital in 
this sense, whereas for automorphisms, there are only three. The corresponding result still holds however, which says that two endomorphisms are 
conjugate if and only if their families of orbitals are order-isomorphic under an isomorphism which carries any orbital to an isomorphic 
orbital, where the isomorphism must respect the action of the map. In the case of automorphisms, the usual formulation is that double homogeneity is 
assumed (any 2-element set can be mapped to any other by an automorphism) but this is needed to build an isomorphism between orbitals of the same 
parity, so if we just postulate the isomorphism, then this is not needed. For generic epimorphisms, there is only one possibility for this isomorphism 
type, at any rate in End($\mathbb Q$), which is why it was relatively easy in \cite{Lockett2} to describe such elements explicitly.

We can give a clearer visualization of what orbitals look like as follows. Consider $f \in {\rm End}(X, \le)$ and $x \in X$, let $Y$ be its orbital. 
We first suppose that $Y$ contains no fixed point of $f$, and without loss of generality, that $x < fx$. Then $x \le fx \le f^2x \le f^3x \le \ldots$,
and in view of our assumption, these inequalities must all be strict. Let us define subsets $I_n$ of $X$ for $n \ge 0$ by $I_n = f^{-n}f^nx$. Thus 
$x \in I_n$. We first notice that $I_n \subseteq I_{n+1}$, for if $y \in I_n$ then $f^ny = f^nx$, so $f^{n+1}y = f^{n+1}x$ giving $y \in I_{n+1}$. 
Next, $I_n$ is convex, since if $y_1 \le z \le y_2$ where $y_1, y_2 \in I_n$ then $f^nx = f^ny_n \le f^nz \le f^ny_2 = f^nx$, so also $f^nz = f^nx$ 
and $z \in I_n$. Next we note that for each $n \ge 0$ and $m \in {\mathbb Z}$, $f^mI_n < f^{m+1}I_n$. To see this, let $y \in f^mI_n$ and 
$z \in f^{m+1}I_n$, and suppose for a contradiction that $z \le y$. First treating the case $m \ge 0$, $y = f^my'$ and $z = f^{m+1}z'$ where 
$y', z' \in I_n$. Hence $f^{m+1}z' \le f^my'$, so $f^{m+n+1}x = f^{m+n+1}z' \le f^{m+n}y' = f^{m+n}x$, which is a contradiction. Next, if $m = -l < 0$, 
then $f^ly, f^{l-1}z \in I_n$, so $f^nx = f^{n+l-1}z \le f^{n+1-l}y = f^{n-1}x$, again a contradiction.

Now we let $I^* = \bigcup_{n \ge 0}I_n$, and it follows that $\ldots < f^{-2}I^* < f^{-1}I^* < I^* < fI^* < f^2I^* < \ldots$. The orbit of $x$ (under 
taking arbitrary forwards and backwards images of $x$ under $f$ is then $\bigcup_{m \in {\mathbb Z}}f^mI^*$, and the orbital is the convex closure of 
this, which is clearly equal to $\bigcup_{m \in {\mathbb Z}}f^mJ^*$, where $J^* = I^* \cup \{y \in X: I^* < y < fI^*\} = (\inf I^*, \inf fI^*)$ or 
$[\inf I^*, \inf fI^*)$ if $\inf I^* \in I^*$. It should be pointed out however that there is a multitude of possible behaviours here, since we are 
not (yet) assuming that $f$ is an epimorphism, so the inverse images may be empty, or there may be a mixture, some empty and some not, so our 
description is insufficient to pin down the structure of the orbitals in any really satisfactory way.

Now let us consider the case in which the orbital contains a fixed point, and we assume this is $x$. Then for $n \ge 0$, $f^{-n}x$ is convex, and 
$\{x\} = f^0x \subseteq f^{-1}x \subseteq f^{-2}x \subseteq f^{-3}x \subseteq \ldots$ and the orbital is $\bigcup_{n \ge 0}f^{-n}x$. This is the case 
which applies when $f$ is a generic endomorphism of $\mathbb Q$ (see \cite{Lockett2}), and then there is just one pattern of behaviour up to 
isomorphism. Each $f^{-(n+1)}x$ extends $f^{-n}x$ above and below, and each point has inverse image isomorphic to $\mathbb Q$. In the general case 
there are many more options, even if $X = {\mathbb Q}$. For instance $x$ may be the greatest point of the orbital, or the least. As another example, 
we can easily exhibit $2^{\aleph_0}$ behaviours for $X = \omega$, by choosing any strictly increasing sequence $n_0 < n_1 < n_2 < \ldots$ in $\omega$, 
and letting $f(m) = 0$ if $m \le n_0$, and $f(m) = n_k$ if $n_k < m \le n_{k+1}$, in which case $f^{-k}0 = [0, n_k]$ for each $k$.

We can specialize these to the cases of embeddings or epimorphisms (the former is equivalent to considering the structure $(X, <)$ in place of 
$(X, \le)$). For embeddings, all $f^nx$ are of size at most 1. For epimorphisms, all $f^nx$ are non-empty. (For $n \ge 0$ these are both true in all 
cases by definition of `function'; it is negative values which are significant here.)

\section{Generic endomorphisms of trees}

In this final section, we expand on the information given in \cite{Lockett2} about generic endomorphisms of trees. We recall that an automorphism is 
said to be {\em generic} if it lies in a comeagre conjugacy class. If we continue to insist (as we must) that the conjugating element is an 
automorphism, then we can also state what it would mean for the five other kinds of maps to be `generic', using precisely the same definition. It was 
shown however in \cite{Lockett2} that there are only three cases which can possibly arise, as (under reasonable conditions on the structure) any 
generic map must be surjective, so there are just three possibilities, automorphisms, bimorphisms, and epimorphisms. The intuition which we mainly 
follow is that a map should be regarded as generic provided it fulfils all possible properties that can be required of it on the basis of a finite map 
of the relevant kind. Thus for instance, it will be surjective, since no finite map can exclude any given element from being added to the range in 
some extension.

The general existence results given in \cite{Lockett2} for generic endomorphisms required that the structure itself be homogeneous, meaning in the 
classical sense, that is, $\IA$-homogeneous, but there are no such trees at all. The solution adopted in \cite{Lockett2} was to work in an expanded 
language where the trees are viewed as semilattices. The types we are principally interested in are the ones described by Droste in \cite{Droste}, which 
are $T_k^+$ and $T_k^-$ for $2 \le k \le \aleph_0$, being countable trees in which all branches are isomorphic to $\mathbb Q$, and with all 
ramification points of ramification order $k$; those of `positive type' are $T_k^+$, which means that the ramification points lie in the structure, 
and those of `negative type', in which none of the ramification points lie in the structure (so they form a dense set of cuts). Note also that since 
the trees situation is a lot more complicated than for linear orders, we only look at the strict relation $<$, though the analysis also applies to
$\le$ under suitable modification. 

Our main goal here is to establish that there are generic members of all three monoids, Aut, Bi, and Epi, which are all distinct, in the case of 
all $T_k^+$ and $T_k^-$, thereby giving an example that we were unable to come up with in \cite{Lockett2}, and we also give some characterizations of 
what these are like, extending the remarks made for Aut there. The main technical tool to do this is Theorem 2.3 from \cite{Lockett2}, and we also use 
ideas described in \cite{Kuske}. In all cases we are viewing $T = T_k^+$ or $T_k^-$ as a lower semilattice, and we consider the family $P$ of finite 
partial isomorphisms, monomorphisms, endomorphisms in the three cases respectively. The import of requiring the structure to be viewed as a lower 
semilattice is that the finite partial maps we consider should be defined on subsemilattices, which are therefore just finite trees. In the positive 
case $T_k^+$ this is all we need to say. In the negative case $T_k^-$, we should view the relevant lower subsemilattice as `coloured', one colour for 
the points of the structure, and the other for the ramification points, and the maps must preserve the colours, as well as the semilattice structure.

According to Theorem 2.3 of \cite{Lockett2}, to guarantee the existence of a generic, it suffices to verify that the natural family of finite partial 
maps $P$ has the `weak amalgamation property' (and joint embedding property), but for this, in practice, we show that $P$ has a cofinal subset with 
the amalgamation property (and another question left open in \cite{Lockett2} was whether this is actually any stronger). In \cite{Kuske} the fact that 
there is a generic automorphism of the countable universal homogeneous partial order was verified using this condition, and it was also remarked how 
the existence of a generic automorphism of $({\mathbb Q}, <)$ may be established using the same method (even though this had previously been shown to 
exist by a direct argument). The family $P$ consists of all pairs $(a, p)$ such that $a$ is in the `age' of the structure, meaning that it is 
isomorphic to a finite substructure, and $p$ is a partial map of $a$ of whichever kind is being considered. Here therefore $a$ is a lower semilattice.

If $f$ is an endomorphism of $T$, then a point $x$ of $T$ is said to {\em spiral} under $f$ if for some $n > m \ge 0$, $f^mx$ and $f^nx$ are 
comparable. The spiral is {\em positive} if $f^mx < f^nx$, {\em negative} if $f^mx > f^nx$, and {\em zero} if $f^mx = f^nx$. The {\em orbit} of $x$ is 
$\{y: (\exists m, n)f^mx = f^ny\}$, and the convex closure of an orbit is called an {\em orbital}. It is easy to see that the `parity' (positive, 
negative, or zero) of an orbital is independent of which of its elements is chosen. In the case of bijections (automorphisms or bimorphisms) the 
definitions can be slightly simplified (by taking $m = 0$); the version given also applies to endomorphisms. The notions also give rise to 
corresponding `partial' notions for members of $P$, such as `partial orbital'. 

\begin{lemma} \label{4.1} Any partial isomorphism (monomorphism, endomorphism) $p$ of a finite semilattice $a$ can be extended to a finite partial 
isomorphism (monomorphism, endomorphism respectively) $q$ of a finite semilattice $b \supseteq a$ such that all points of the domain of $q$ spiral.
\end{lemma}

\begin{proof} Note that the situation here is different from the case for partial orders, see \cite{Kuske}, but the counter-example given there is 
definitely {\em not} semilinear, so there is no conflict. Note further that the lemma covers both cases, coloured and not, with the provisos on the 
maps mentioned above, though we mainly concentrate on the monochromatic case (positive type). It suffices to extend a given $(a, p)$ in which some 
point does not spiral, to $(b, q)$, in such a way that the number of orbits of points which do not spiral properly decreases, and then we may repeat 
until none remain. So suppose that $x \in a$ does not spiral under $p$, and let $x$ be chosen so that it is minimal with this property. This means 
that $\{x, px, p^2x, \ldots, p^{n-1}x\}$ is an antichain (pairwise incomparable) for some $n$ (which may be 1), such that $p^{n-1}x$ does not lie in 
dom $p$. Furthermore, by considering inverse images as necessary, we may also suppose that $x$ does not lie in range $p$.

We now observe that the subsets $a_i = \{y \in a: p^ix \le y\}$ are pairwise disjoint subtrees of $a$ (indeed subsemilattices). For if 
$y \in a_i \cap a_j$ where $i < j$, then as $a$ is a tree, $p^ix$ and $p^jx$ are comparable, and so $x$ spirals after all. We now need to extend $a$ 
and $p$. First consider the case of Aut. Here we can extend $a$ so that all $a_i$ are isomorphic by isomorphisms which respect levels, and agree with 
$p$ where defined, so by extending further and increasing $n$ if necessary, $p$ actually is an isomorphism from $a_i$ to $a_{i+1}$ for 
$0 \le i < n-1$ and $a_{n-1} \cap {\rm dom} \, p = \emptyset$. We can now extend $p$ to $q$, by requiring that $q^n$ fix all elements of $a_0$. This 
ensures that all elements of orbits contained in $\bigcup_{i < n}a_i$ spiral (with parity 0), so the number of orbits of points which do not spiral 
has decreased.

If there is no point of the domain of $p$ below any $p^ix$, then the sketch in the above paragraph suffices, as the map $q$ will be order-preserving. 
If however there is some point $y$ below some $p^ix$, then the argument may need modification. By choice of $x$, such points must spiral, and we take 
$y$ to be maximal such. Suppose that its spiral is zero or negative, and let $p^ly \le p^ky$ where $l > k \ge 0$ and $l-k$ is minimal. Extending $p$ 
if necessary, we may suppose that $k = 0$ and $y < x$ (this is the only reason why we have taken the spiral non-positive; if it is positive, then some 
image of $y$ will be taken below $x$ and above $y$ instead). By extending, we suppose that $n$ is a multiple of $l$, and now one can see that $q$ is 
order-preserving. The main case is that $p^{n-1}y < p^{n-1}x$ and $p^ny \le y$ (as $l$ divides $n$) $< x = p^nx$.

The above argument has to be modified for Bi and Epi. Once again we extend $p$ (one point at a time) and the subtrees $a_i$ for $0 \le i < n-1$, so 
that they are all contained in the domain of $p$, and by increasing $n$ if necessary, also assume that $a_{n-1}$ is linearly ordered and disjoint 
from the domain of $p$. (This is possible, and indeed we cannot `avoid' the possibility, since endomorphisms are allowed to take incomparable points 
to comparable ones.) We can now complete to spirals using a similar idea to that for isomorphisms; if there is no point below $p^nx$ in the domain of 
$p$, we just let the extension $q$ fix all points of $a_n$. If however there is such a maximal point $y$, by assumption it spirals with length $l - k$ 
as before, and we add new points in chains $a_m$ for $n \le m < n + l - k$ in bijective correspondence with those of $a_n$, and let $q$ act as an 
isomorphism from $a_m$ to $a_{m+1}$ for $n \le m < n + l - k - 1$, and from $a_{n + l - k - 1}$ to $a_n$.
\end{proof}

\begin{theorem} \label{4.2}
There are generic homomorphisms of each tree $T$ of the form $T_k^+$ and $T_k^-$ for $2 \le k \le \aleph_0$ for each of the three monoids Aut, Bi, and 
Epi, which are all distinct. 
\end{theorem}

\begin{proof} By \cite{Lockett2} Theorem 2.3 it suffices to show that $P$ has a cofinal subset $P_1$ with the amalgamation property. In each case we 
take this to comprise all those partial endomorphisms $p$ of finite subtrees of $T$ (in the case of $T_k^-$ which preserve `colours') all points of 
whose domains spiral, and such that whenever $x < y$ lie in distinct partial orbitals of $p$, and they have the same parity, there is some $z$ such 
that $x < z < y$ lying in an orbital of $p$ of different parity from $x$ and $y$.

The fact that $P_1$ is cofinal in $P$ follows mainly from Lemma \ref{4.1}---we may extend any given member of $P$ to some $p$ such that all members of 
its domain spiral. Then one extends further to ensure the given condition by adding in extra orbitals appropriately. If $x$ lies in a positive orbital 
for instance, where $x < y$ are given, the easiest way to ensure that $z$ as desired exists, and that all the conditions are compatible, is to let 
$x < z < y$, and let it have the same spiral length as $x$, but opposite parity.

Finally we have to check that $P_1$ has the amalgamation property. Since $P_1$ is cofinal in $P$, it suffices to show that if $(a, p)$ in $P_1$ has two 
extensions each obtained by adding just one orbit, $(b, q)$ and $(c, r)$, then these can be amalgamated in $P$. Let $q$ and $r$ be obtained by 
adjoining the orbits of $x$, $y$ respectively. If the added orbitals lie between different orbitals of $a$, then we can just take the union. If they 
lie between the same orbitals of $a$, then we have to decide the relationship between the members of the orbits of $x$ and $y$, but this can be done 
arbitrarily, for instance, putting the points of the orbit of $x$ below those of $y$ when they can be made comparable.

The point of requiring $(a, p)$ to lie in $P_1$ is that the extensions $(b, q)$ and $(c, r)$ cannot require the points of the new orbits to be related 
to points of $a$ in `different' ways, which would be the obstacle to amalgamation.   \end{proof}

The theorem is a little unsatisfactory in the sense that it does not describe in any at all explicit way what the generic maps of the various kinds 
look like. By contrast, in \cite{Lockett2} we outlined what they would be like in the case of automorphisms (though without formally proving 
existence). If the ramification order is infinite, then one can see that some branch (maximal chain) will be fixed (setwise); this is because no finite 
partial automorphism can prevent higher and higher points being fixed, so one can imagine starting from such a fixed branch, and `attaching' the 
generic behaviour to it, and the same applies to bimorphisms and epimorphisms. For a generic automorphism or epimorphism, the restriction to the fixed 
branch will itself be a generic automorphism of (that copy of) $\mathbb Q$, but this is not true for the bimorphism (since it will not be surjective). 
Off the fixed branch the behaviour will exhibit spirals of all possible finite lengths, and each of these will occur infinitely many times, but a 
completely explicit description seems too complicated to be of any value. If $k$ is finite, then there will definitely not be any (setwise) fixed 
branches in the generic maps, since a finite map suffices to exclude this possibility.

Similar results also apply for the reflexive relation $\le$, though with more complicated details.

\end{document}